\documentclass[a4paper]{amsart}

\usepackage{amssymb}
\usepackage{stmaryrd}

\makeatletter
\def\section{\@startsection{section}{1}%
  \z@{.7\linespacing\@plus\linespacing}{.5\linespacing}%
  {\normalfont\bfseries\centering}}
\def\@secnumfont{\bfseries}
\makeatother

\renewcommand{\o}{\circ}
\def\frak{\mathfrak}
\def\Bbb{\mathbb}
\def\Cal{\mathcal}

\newcommand{\al}{\alpha}

\newcommand{\la}{\lambda}

\newcommand{\om}{\omega}

\newcommand{\si}{\sigma}
\renewcommand{\th}{\theta}
\newcommand{\ze}{\zeta}

\newcommand{\Ga}{\Gamma}

\newcommand{\Rho}{{\mbox{\sf P}}}
\newcommand{\Om}{\Omega}

\newcommand{\ad}{\operatorname{ad}}
\newcommand{\Ad}{\operatorname{Ad}}

\renewcommand{\exp}{\operatorname{exp}}

\newcommand{\tr}{\operatorname{tr}}

\renewcommand{\o}{\circ}

\newcommand{\pmat}[1]{\begin{pmatrix}#1\end{pmatrix}}

\let\[=\llbracket
\let\]=\rrbracket
\let\x=\times

\def\g{\frak g}
\def\p{\frak p}

\def\X{\frak X}
\def\({\big(}
\def\){\big)}
\def\R{\Bbb R}
\def\C{\Bbb C}

\newcommand{\tg}{{\tilde{\frak g}}}
\newcommand{\tp}{{\tilde{\frak p}}}

\def\L{{\Cal L}}
\def\G{{\Cal G}}
\def\Ce{{\Cal C}}
\def\.{\hbox to5pt{\hss$\cdot$\hss}}
\def\span#1{\langle#1\rangle}
\renewcommand{\P}{\Bbb P}
\newcommand{\Proj}{\operatorname{\Cal P}}

\def\CE{\hat\Ce}
\def\Nabla{\hat\nabla}
\def\OM{\hat\Omega}
\def\Al{\hat\alpha}
\def\E{\hat E}
\def\h{\hat h}

\def\ins{\operatorname{\lrcorner}}
\def\bez#1{\setminus\{#1\}}
\def\ddt#1{\frac{d}{dt}\ifx#1\nic\else\big\vert_{#1}\fi}

\newtheorem*{prop*}{Proposition}

\newtheorem*{thm*}{Theorem}

\newtheorem*{lem*}{Lemma}

\newtheorem*{cor*}{Corollary}

\begin{document}

\title{Remarks on Special Symplectic Connections}
\author{Martin Pan\'ak and Vojt\v ech \v Z\'adn\'\i k}

\thanks{The first author was supported by the grant nr. 201/05/P088, the second author 
by the grant nr. 201/06/P379, both grants of the Grant Agency of the Czech Republic.}
\address{Masaryk University, Brno, Czech republic}
\email{naca@math.muni.cz, zadnik@math.muni.cz}
%\date{\today, version 4}

\begin{abstract}
The notion of special symplectic connections is closely related to 
parabolic contact geometries due to the work of M. Cahen and L. Schwachh\"ofer.
We remind their characterization and reinterpret the result in terms of generalized Weyl connections.
The aim of this paper is to provide an alternative and more explicit construction 
of special symplectic connections of three types from the list.
This is done by pulling back an ambient linear connection from the total space of a 
natural scale bundle over the homogeneous model of the corresponding parabolic contact 
structure.
\end{abstract}

\maketitle

%\setcounter{tocdepth}{1}
%\tableofcontents

%%%
\section{Introduction}			\label{1}
Special symplectic connection on a symplectic manifold $(M,\om)$ is a 
torsion--free linear connection preserving $\om$ which is special in the sense 
of definitions in \ref{2.1}.
The definition of special symplectic connection is rather wide,
however, there is a nice link between special symplectic connections and 
parabolic contact geometries, which was established in the profound 
paper \cite{CS}.
The main result of that paper states that, locally, any special symplectic 
connection on $M$ comes via a symplectic reduction
from a specific linear connection on a one--dimension bigger contact 
manifold $\Ce$, the homogeneous model of some parabolic contact geometry.
All the necessary background on parabolic contact geometries is collected
in section \ref{2}.
The construction and the characterization from \cite{CS} is quickly reminded
in section \ref{3}, culminating in Theorem \ref{th1}.

In the next section we provide an alternative and rather direct approach to
special symplectic connections.
Firstly we reinterpret the previous results in terms of
parabolic geometries so that the specific linear connections on
$\Ce$ are exactly the exact Weyl connections corresponding to specific
choices of scales.
A~choice of scale further defines a bundle projection from $T\Ce$ to the 
contact distribution $D\subset T\Ce$ and this gives rise to a partial 
contact connection on $D$.
By the very construction, the only ingredient which yields the special
symplectic connection on $M$ is just the partial contact connection
associated to the choice of scale, Proposition \ref{th2}.
  
Finally, the direct construction of special symplectic connections works via
a pull--back of an ambient symplectic connection on the total space of a 
canonical scale bundle $\CE\to\Ce$. 
Namely for several specified cases we can find a convenient ambient
connection on $\CE$ and then compare the exact Weyl 
connection and the pull--back connection on $\Ce$ corresponding to the 
choice of scale so that they coincide on the contact distribution $D$, 
Theorem \ref{th3}.
By the previous results, they give rise the same symplectic connection on $M$ 
after the reduction.
This construction applies to the projective contact structures, CR structures 
of hypersurface type, and Lagrangean contact structures, which are dealt
in subsections \ref{4.1}, \ref{4.2}, and \ref{4.3}, respectively.

\subsection{Special symplectic connections}	\label{2.1}
Given a smooth manifold $M$ with a symplectic structure $\om\in\Om^2(M)$,
linear connection $\nabla$ on $M$ is said to be \textit{symplectic} if it is torsion
free and $\om$ is parallel with respect to $\nabla$. 
There is a lot of symplectic connections to a given symplectic structure,
hence studying this subject, further restrictive  conditions appear.
Following the article \cite{CS}, we consider the \textit{special symplectic
connections} defined as symplectic connections belonging to some of the
following classes:

(i) \textit{Connections of Ricci type}. The curvature tensor of a symplectic connection 
decomposes under the action of the symplectic group into two irreducible components. 
One of them corresponds to the Ricci curvature and the other one is the 
Ricci--flat part.  If the curvature tensor consists only
of the Ricci curvature part, then the connection is said to be of Ricci type.

(ii) \textit{Bochner--K\"ahler connections}. Let the symplectic form be the
K\"ahler form of a (pseudo--)K\"ahler metric and let the connection preserve this
(pseudo--)K\"ahler structure. The curvature tensor decomposes similarly as in the previous case into two parts
but this time under the action of the (pseudo--)unitary group. These are called  Ricci curvature and Bochner curvature.
If the Bochner curvature vanishes, the connection is  called Bochner--K\"ahler.

(iii) \textit{Bochner--bi--Lagrangean connections}. 
A~bi--Lagrangean structure on a symplectic manifold consists of two complementary 
Lagrangean distributions. If a symplectic connection preserves such structure,
i.e. both the Lagrangean distributions are parallel, then again the curvature tensor
decomposes into the Ricci and Bochner part. If the Bochner curvature
vanishes, we speak about Bochner--bi--Lagrangean connections.

(iv) \textit{Connections with special symplectic holonomies}.
We say that a symplectic connection has special symplectic holonomy if its holonomy is
contained in a proper absolutely irreducible subgroup of the symplectic group.
Special symplectic holonomies are completely classified and studied by various
people.

\smallskip
Connections of Ricci type are characterized in the interesting article \cite{CGS}, 
see remark \ref{4.1}(a) for some detail.
The Bochner--K\"ahler metrics (marginally also the Bochner--bi--Lagrangean
structures) have been thoroughly studied in the deep article \cite{B}.
See also \cite{P-S} for further investigation of the subject which
is more relevant to our recent interests.
For more remarks and references on special symplectic connections we generally 
refer to \cite{CS}.

Note that all the previous definitions admit an analogy in complex/holomorphic
setting but we are dealing only with the real structures in this paper.

\subsection*{Acknowledgements} 
We would like to thank in the first place to Andreas \v Cap  for the 
fruitful discussions and suggestions concerning mostly the Weyl structures,
especially in the technical part of \ref{4.1}. 
Among others we would like to mention Lorenz Schwachh\"ofer, Jan Slov\' ak 
and Ji\v r\'\i\ Van\v zura, who were willing to discuss some aspects of the 
geometries in this article.

\section{Parabolic contact geometries and Weyl connections} \label{2} 

In this section we provide the necessary background from parabolic geometries
and generalized Weyl structures as can be found in \cite{Slo}, \cite{C-S} or,
the most comprehensively, in \cite{parabook}.

\subsection{Parabolic contact geometries}	\label{2.2}
Semisimple Lie algebra admits a \textit{contact grading} if there is a grading 
$\g=\g_{-2}\oplus\g_{-1}\oplus\g_0\oplus\g_1\oplus\g_2$ such
that $\g_{-2}$ is one dimensional and the Lie bracket $[\ ,\ ]:\g_{-1}\x\g_{-1}
\to\g_{-2}$ is non--degenerate.
If $\g$ admits a contact grading, then $\g$ has to be simple.
Any complex simple Lie algebra, except $\frak{sl}(2,\C)$, admits a unique 
contact grading, but this is not guaranteed generally in real case.
However, the split real form of complex simple Lie algebra and most of 
non--compact non--complex real Lie algebras admit a contact grading.

Let $\g$ be a real simple Lie algebra admitting a contact grading, let
$\p:=\g_0\oplus\g_1\oplus\g_2$ be the corresponding parabolic subalgebra,
and let $\p_+:=\g_1\oplus\g_2$.
Let further $\frak{z}(\g_0)$ be the center of $\g_0$.
Let $E\in\frak z(\g_0)$ be the grading element of $\g$ and let
$\g_0'\subset\g_0$ be the orthogonal complement of $E$ with respect to the
Killing form on $\g$.
From the invariance of the Killing form and the fact that
$[\g_{-2},\g_2]=\span{E}$, the subalgebra $\g_0'\subset\g_0$ is
equivalently characterized by the fact that $[\g_0',\g_2]=0$.
For later use let us denote $\p':=\g_0'\oplus\p_+$.%
\footnote{Note that in \cite{CS} the essential subalgebras $\g_0'$ and $\p'$ are 
denoted by $\frak{h}$ and $\p_0$, respectively.}

For a semisimple Lie group $G$ and a parabolic subgroup
$P\subset G$, \textit{parabolic geometry} of type $(G,P)$ on a smooth manifold 
$M$ consists of a principal $P$--bundle $\G\to M$ and a Cartan connection 
$\eta\in\Om^1(\G,\g)$, where $\g$ is the Lie algebra of $G$.
If $\g$ is simple Lie algebra admitting a contact grading and the Lie 
subalgebra $\p\subset\g$ of $P$ corresponds to this grading, then we 
speak about \textit{parabolic contact geometry}.
The contact grading of $\g$ gives rise to a contact structure on $M$ as
follows.
Under the usual identification $TM\cong\G\x_P\g/\p$ via $\eta$, 
the $P$--invariant subspace $(\g_{-1}\oplus\p)/\p\subset\g/\p$, 
defines a distribution $D\subset TM$, namely
\begin{equation}			\label{eqD}
  D\cong \G\x_P(\g_{-1}\oplus\p)/\p.
\end{equation}
For regular parabolic geometries of these types, the distribution $D\subset TM$
defined by \eqref{eqD} is a contact distribution.
The Lie bracket of vector fields induces the so--called \textit{Levi bracket} 
on the associated graded bundle $\operatorname{gr}(TM)=D\oplus TM/D$, which 
is an algebraic bracket of the form $\L:D\wedge D\to TM/D$.
The regularity means the Levi bracket corresponds to the Lie bracket on 
$\g_-=\g_{-1}\oplus\g_{-2}$.

Any contact distribution can be always given as the kernel of a \textit{contact form}
$\th\in\Om^1(M)$, i.e. a one--form such that $\th\wedge (d\th)^n$ is a volume
form on $M$. 
In particular, the restriction of $d\th$ to $D\wedge D$ is non--degenerate.
For any choice of contact form $\th$, let $r_\th\in\X(M)$ be the corresponding
\textit{Reeb vector field}, i.e. the unique vector field on $M$ satisfying
$r_\th\ins d\th=0$ and $\th(r_\th)=1$.
This further provides a trivialization of the quotient bundle $TM/D$ so that 
$TM\cong D\oplus\R$. 
Next, if $X$ and $Y$ are sections of $D=\ker\th$ then $d\th(X,Y)=-\th([X,Y])$
by the definition of exterior differential.
Altogether, under the trivialization above, the restriction of $d\th$ to 
$D\wedge D$ coincides with the Levi bracket $\L$ up to the sign.

\subsection{Weyl structures}
Let $(\G\to M,\eta)$ be a parabolic geometry of type $(G,P)$.
Let $\p\subset\g$ be the Lie algebras of the Lie groups $P\subset G$ and let
$\g=\g_{-k}\oplus\dots\oplus\g_0\oplus\dots\oplus\g_k$ be the corresponding
grading of $\g$.
Let $G_0\subset P$ be the Lie group with Lie algebra $\g_0$ and let
$P_+:=\exp\p_+$ so that $P=G_0\rtimes P_+$.
Let further $\G_0:=\G/P_+\to M$ be the underlying $G_0$--bundle and let 
$\pi_0:\G\to\G_0$ be the canonical projection. 
The filtration of the Lie algebra $\g$ gives rise to a filtration of $TM$ and 
the principal $G_0$--bundle $\G_0\to M$ plays the role of the frame bundle of 
the associated graded $\operatorname{gr}(TM)$. 
The reduction of the structure group of $TM$ to $G_0$ often corresponds to an
additional geometric structure on $M$ and this collection of data we 
call the \textit{underlying structure} on $M$ 
(see e.g.\ \cite{C-S} for more precise formulations).

A~{\it Weyl structure} for the parabolic geometry $(\G\to M,\eta)$ is a global 
smooth $G_0$--equivariant section $\sigma:\G_0\to\G$ of the projection $\pi_0$.
In particular, any Weyl structure provides a reduction of the $P$--principal bundle 
$\G\to M$ to the subgroup $G_0\subset P$.
Denote by $\eta_i$ the $\g_i$--component of the Cartan connection
$\eta\in\Om^1(\G,\g)$. 
For a Weyl structure $\si:\G_0\to\G$, the pull--back $\sigma^*\eta_0$ defines 
a principal connection on the principal bundle $\G_0$;
this is called the {\it Weyl connection} of the Weyl structure $\sigma$.
Next, the form $\si^*\eta_-\in\Om^1(\G_0,\g_-)$ provides an identification of 
the tangent bundle $TM$ with the associated graded tangent bundle 
$\operatorname{gr}(TM)$
and the form $\si^*\eta_+\in\Om^1(\G_0,\p_+)$ is called the \textit{Rho--tensor}, 
denoted by $\Rho^\si$.
The Rho--tensor is used to compare the Cartan connection $\eta$ on $\G$ and the
principal connection on $\G$ extending the Weyl connection $\si^*\eta_0$ from
the image of $\si:\G_0\to\G$.

Any Weyl connection induces connections on all bundles associated to $\G_0$,
in particular, there is an induced linear connection on $TM$.
By definition, any Weyl connection preserves the underlying structure on $M$.  
On the other hand, 
there are particularly convenient bundles such that the induced
connection from $\si^*\eta_0$ is sufficient to determine whole the Weyl
structure $\si$.
These are the so--called bundles of scales, the oriented line bundles over $M$
defined as follows.

\subsection{Scales and exact Weyl connections}		\label{3.2}
Let $\L\to M$ be a principal $\R_+$--bundle associated to $\G_0$. 
This is determined by a group homomorphism $\la:G_0\to\R_+$ whose derivative
is denoted by  $\la':\g_0\to\R$.
The Lie algebra $\g_0$ is reductive, i.e. $\g_0$ splits into a direct sum of the
center $\frak{z}(\g_0)$ and the semisimple part, hence the only elements that
can act non--trivially by $\la'$ are from $\frak{z}(\g_0)$.
Next, the restriction of the Killing form $B$ to $\g_0$ and further to
$\frak{z}(\g_0)$ is non--degenerate.
Altogether, for any representation $\la':\g_0\to\R$ there is a unique element
$E_\la\in\frak{z}(\g_0)$ such that 
\begin{equation}			\label{eqla}
  \la'(A)=B(E_\la,A)
\end{equation}
for all $A\in\g_0$.
By Schur's lemma, $E_\la$ acts by a real scalar on any irreducible 
representation of $G_0$.
An element $E_\la\in\frak{z}(\g_0)$ is called a {\it scaling element} if 
it acts by a non--zero real scalar on each $G_0$--irreducible component of 
$\p_{+}$.
(In general, the grading element of $\g$ is a scaling element.)
A~{\it bundle of scales} is a principal $\R_+$--bundle associated 
to $\G_0$ via a homomorphism $\la:G_0\to \R_{+}$, whose derivative 
is given by \eqref{eqla} for some scaling element $E_{\la}$.
Bundle of scales $\L^\la\to M$ corresponding to $\la$ is naturally identified 
with $\G_0/\ker\la$, the orbit space of the action of the normal subgroup
$\ker\la\subset G_0$ on $\G_0$.

Let $\L^\la\to M$ be a fixed bundle of scales and let $\si:\G_0\to\G$ be a Weyl
structure of a parabolic geometry $(\G\to M,\eta)$.
Then the Weyl connection $\si^*\eta_0$ on $\G_0$ induces a principal
connection on $\L^\la$ and \cite[Theorem 3.12]{C-S} shows that this
mapping establishes a bijective correspondence between the set of Weyl 
structures and the set of principal connections on $\L^\la$.
Note that the surjectivity part of the statement is rather implicit, however
there is a distinguished subclass of Weyl structures which allow more
satisfactory interpretation, namely the \textit{exact Weyl strucures} defined 
as follows.
Any bundle of scales is trivial and so it admits  global smooth sections,
which we usually refer to as \textit{choices of scale}.
Any choice of scale gives rise to a flat principal connection on $\L^\la$ and 
the corresponding Weyl structure is then called {\it exact}.

Furthermore, due to the identification $\L^\la=\G_0/\ker\la$, 
the sections of $\L^\la\to M$ are in a bijective correspondence with 
reductions of the principal bundle $\G_0\to M$ to the structure group 
$\ker\la\subset G_0$.
Altogether for any choice of scale, the composition of the two reductions
above is a reduction of the principal $P$--bundle $\G\to M$ to the structure 
group $\ker\la\subset G_0\subset P$; let us denote the resulting bundle by $\G_0'$.
Hence the corresponding exact Weyl connection has holonomy in $\ker\la$
and by general principles from the theory of $G$--structures, it preserves the 
geometric quantity corresponding to the choice of scale.

In the cases of parabolic contact geometries, the canonical candidate for the 
bundle of scales is the bundle of positive contact one--forms.
Note that this is the bundle of scales corresponding to (a non--zero multiple of) 
the grading element $E\in\frak{z}(\g_0)$, hence the Lie subalgebra
$\ker\la'\subset\g_0$ is identified with $\g_0'$ from \ref{2.2}.
Let $G_0'$ be the connected subgroup in $G$ corresponding to $\g_0'\subset\g$.
Reinterpreting the general principles above:
the choice of a contact one--form $\th\in\Om^1(M)$ yields a reduction
$\G_0'\subset\G$ of the principal bundle $\G\to M$ to the subgroup $G_0'\subset P$ 
and a principal connection on $\G_0'$, which 
preserves not only the underlying structure on $M$
(so in particular the contact distribution $D=\ker\th$), but moreover 
the form $\th$ itself.
In other words, $\th$ is parallel with respect to the induced linear
connection on $TM$.

\section{Characterization of special symplectic connections}	\label{3}

In this section the quick review of the construction of the special symplectic connections from the article
\cite{CS} is described.
Consult e.g. \cite{Chu} for details on invariant symplectic structures on
homogeneous spaces.

\subsection{Adjoint orbit and its projectivization}	\label{2.3}
Let $\g$ be a real simple Lie algebra admitting  a contact grading and 
let $e^2_+\in\g$ be a maximal root element, i.e.  a generator of $\g_2$.
Let $G$ be a connected Lie group with Lie algebra $\g$. 
Consider the adjoint orbit of $e_+^2$ and its oriented projectivization:
\begin{equation}			\label{eqcone}
\CE:=\Ad_G(e^2_+)\subset\g, \ 
\Ce:=\Proj^o(\CE)\subset\Proj^o(\g).
\end{equation}
The restriction of the natural projection $p:\g\bez0\to\Proj^o(\g)$ to $\CE$
yields the principal $\R_+$--bundle $p:\CE\to\Ce$, which we call 
\textit{the cone}.
The right action of $\R_+$  is just the multiplication by positive real 
scalars.
The fundamental vector field of this action is the
\textit{Euler vector field} $\E$ defined as $\E(x):=x$, 
for any $x\in\CE\subset\g$. %, in particular, $\Fl^{\E}_t(x)=e^t x$.

Since $\CE$ is an adjoint orbit of $G$ in
$\g$, and $\g$ can be identified with $\g^*$ via the Killing form,
there is a canonical $G$--invariant symplectic form $\OM$ on $\CE$.
For any $X,Y\in\g$ and $\al\in\CE\subset\g^*$, the value of $\OM$ is given 
by the formula 
\begin{equation*}
\OM(\ad^*_X(\al),\ad^*_Y(\al)):=\al([X,Y]),
\end{equation*}
where $\ad^*:\g\to\frak{gl}(\g^*)$ is the infinitesimal coadjoint 
representation and $\ad^*_X(\al)=-\al\o\ad_X$ is viewed as an
element of $T_\al\CE$. Under the identification $\g\cong\g^*$ the previous 
formula reads as
\begin{equation}			\label{eqOm}
\OM(\ad_X(a),\ad_Y(a))=B(a,[X,Y]),
\end{equation}
for any $X,Y\in\g$ and $a\in\CE\subset\g$, where $B:\g\x\g\to\R$ is the
Killing form.
The Euler vector field and the canonical symplectic form defines a (canonical)
$G$--invariant one--form $\Al$ on $\CE$ by 
\begin{equation}			\label{eqAl}
\Al:=\frac12\E\ins\OM.
\end{equation}
Immediately from definitions it follows that $\L_{\E}\OM=2\OM$ and
consequently $d\Al=\OM$.

\begin{lem*}
Let $p:\CE\to\Ce$ be the cone defined by \eqref{eqcone} and let
$P'\subset P$ be the connected subgroups in $G$ corresponding to the
subalgebras $\p'\subset\p\subset\g$ from \ref{2.2}.
Then $\CE\cong G/P'$ and $\Ce\cong G/P$ so that the contact distribution 
$D\subset T(G/P)$ is identified with $Tp\.\ker\Al\subset T\Ce$.
\end{lem*}
\begin{proof}
By definition, the group $G$ acts transitively both on $\CE$ and $\Ce=\CE/\R_+$.
Since $[A,e^2_+]=0$ if and only if $A\in\p'$ and we assume the Lie subgroup
$P'\subset G$ corresponding to $\p'\subset\g$ is connected, the stabilizer of 
$e^2_+$ is precisely $P'$. Hence the orbit $\CE$ is identified with the 
homogeneous space $G/P'$. 
Since $P\supset P'$ is also connected, $P/P'$ is identified with the
subgroup $\{\exp tE:t\in\R\}\cong\R_+$ in $P$.
Hence $P$ preserves the ray of positive multiples of $e^2_+$ so that 
$\Ce=\CE/\R_+$ is identified with $G/P$.

For the last part of the statement, note that the Euler vector field 
is generated by (a non--zero multiple of) the grading element 
$E\in\frak{z}(\g_0)$. 
The canonical one--form $\Al$ on $\CE$ is $G$--invariant, so it is 
determined by its value in the origin $o\in G/P'$, i.e. 
$e^2_+\in\CE$, which is a $P'$--invariant one--form $\phi$ on $\g/\p'$.
By \eqref{eqOm} and \eqref{eqAl}, $\phi$ is explicitly given as 
$\phi(X)=B(e^2_+,[E,X])$, possibly up to a non--zero scalar multiple.
The formula is obviously independent of the representative of $X$ in $\g/\p'$
and the kernel of $\phi$ is just $(\g_{-1}\oplus\p)/\p'$.
The tangent map of the projection $p:\CE\to\Ce$ corresponds to the natural
projection $\g/\p'\to\g/\p$, hence $Tp\.\ker\Al\subset T\Ce$ corresponds to 
$(\g_{-1}\oplus\p)/\p\subset\g/\p$ which defines the contact distribution 
$D\subset T(G/P)$ in \eqref{eqD}.
\end{proof}

\subsection*{Remarks}
(a) Note that in contrast to the definition of the cone in \cite{CS} we do not
assume the center of $G$ is trivial.
Hence the two approaches differ by a (usually finite) covering. 
Because of the very local character of all the constructions that follow,
this causes no problem and we will not mention the difference below.

(b)
The homogeneous space $\CE\cong G/P'$ is an example of a symplectic 
homogeneous space, i.e. a homogeneous space with an invariant symplectic 
structure.
According to \cite[Corollary 1]{Chu}, for $G$ being semisimple,  
any simply connected symplectic homogeneous space of a Lie group $G$ is 
isomorphic to a covering of some $G$--orbit in $\g$, which is thought with 
the (restriction of the) canonical symplectic form.
Moreover the covering map and hence the orbit are unique.

(c) 
According to  \cite[Prop. 3.2]{CS}, the bundle $\CE\to\Ce$ can be
identified with the bundle of positive contact forms on $\Ce$ so that 
$\OM=d\Al$ corresponds to the restriction of the canonical symplectic form on
the cotangent bundle $T^*\Ce$. 
In detail, a section $s:\Ce\to\CE$ yields the contact one--form 
$\th_s:=s^*\Al$ and, by the naturality of the exterior differential, 
$d\th_s=s^*\OM$.

\subsection{General construction}		\label{2.4}
Let $a$ be an element of a real simple Lie algebra $\g$ admitting a contact 
grading.
With the notation as before, let $\xi_a$ be the fundamental vector field of 
the left action of $G$ on $\Ce\cong G/P$ corresponding to $a\in\g$.
Let us denote by $\Ce_a$ the (open) subset in $\Ce$ where $\xi_a$ is 
transverse to the contact distribution $D\subset T\Ce$ and oriented in 
accordance with a fixed orientation of $T\Ce/D$.
The vector field $\xi_a$ gives rise to a unique contact one--form 
$\th_a$  on $\Ce_a$ such that $\xi_a$ is its Reeb field.
In other words, $\th_a\in\Om^1(\Ce_a)$ is uniquely determined by the
conditions 
\begin{equation}			\label{eqth}
  \ker\th_a=D \mathrm{\ and\ } \th_a(\xi_a)= 1.
\end{equation}
Since $\xi_a$ is a contact symmetry, i.e. $\L_{\xi_a}D\subset D$, it easily
follows that $\L_{\xi_a}\th_a=0$ and consequently $\xi_a\ins d\th_a=0$.
Let $T_a\subset G$ denote the one--parameter subgroup corresponding to the
fixed element $a\in\g$.
We say that an open subset $U\subset\Ce_a$ is \textit{regular} if the local 
leaf space $M_U:=T_a\setminus U$ is a manifold.
Since $\xi_a\ins d\th_a=0$ and $d\th_a$ has maximal rank, it descends to a 
symplectic form $\om_a$ on $M_U$, for any regular $U\subset\Ce_a$.

Next, let $\pi:G\to G/P\cong\Ce$ be the canonical $P$--principal bundle and 
consider its restriction to $\Ce_a$.
If $\Ce_a$ is non--empty, then \cite[Theorem 3.4]{CS} describes explicitly a 
subset $\Ga_a$ in $\pi^{-1}(\Ce_a)\subset G$,  which forms a $G_0'$--principal
bundle over $\Ce_a$ where $G_0'$ is the subgroup of $P$ as in \ref{3.2}.
For a regular open subset $U\subset\Ce_a$, denote
$\Ga_U:=\pi^{-1}(U)\subset\Ga_a$.
Note that $\Ga_U$ is invariant under the action of $T_a$.
Denoting $B_U:=T_a\setminus\Ga_U$, 
$B_U\to M_U$ is a $G_0'$--principal bundle and \cite[Theorem 3.5]{CS} 
shows that the restriction of the $(\g_{-2}\oplus\g_{-1}\oplus
\g_0')$--component of the Maurer--Cartan form $\mu\in\Om^1(G,\g)$ to $\Ga_U$ 
descends to a $(\g_{-1}\oplus\g_0')$--valued coframe on $B_U$.
Altogether, the bundle $B_U\to M_U$ is interpreted as a classical
$G_0'$--structure and the $\g_0'$--part of the coframe above induces a linear
connection on $M_U$.
It turns out this connection is special symplectic connection with respect 
to the symplectic form $\om_a$.

Surprisingly, \cite[Theorem B]{CS} proves that any special symplectic 
connection can be at least locally obtained by the previous construction.
With an assumption on $\dim\g\ge 14$, which is equivalent to $\dim
M_U\ge 4$, we reformulate the main result of \cite{CS} as follows.

\begin{thm*}[\cite{CS}]			\label{th1}
Let $\g$ be a simple Lie algebra of dimension $\ge 14$ admitting a contact 
grading.
With the same notation as above, let $a\in\g$ be such that $\Ce_a\subset\Ce$
is non--empty and let $U\subset\Ce_a$ be regular. 
Then 

(a) the local quotient $M_U$ carries a special symplectic connection,

(b) locally, connections from (a) exhaust all the special symplectic connections.
\end{thm*}

An instance of the correspondence between the various classes of special 
symplectic connections and contact gradings of simple Lie algebras is as 
follows.
For $\dim M_U=2n$, special symplectic connections of type (i), (ii) and (iii),
according to the definitions in \ref{2.1}, corresponds to the contact grading
of simple Lie algebras $\frak{sp}(2n+2,\R)$, $\frak{su}(p+1,q+1)$ with 
$p+q=n$, and $\frak{sl}(n+2,\R)$, respectively.
The corresponding parabolic contact structure on $\Ce\cong G/P$ is the 
projective contact structure, CR structure of hypersurface type, and 
Lagrangean contact structure, respectively. 
Details on each of these structures are  treated in the next section in 
details.

\section{Alternative realization of special symplectic connections} \label{4}

Below we describe parabolic contact structures corresponding to
special symplectic connections of type (i), (ii) and (iii) as mentioned above.
The aim of this section is, for each of the listed cases, to provide the
characterization of Theorem \ref{th1}, and so the realization of special
symplectic connections, in more explicit and satisfactory way.
For this purpose we  interpret the model cone $p:\CE\cong G/P'\to G/P\cong\Ce$ 
in each particular case and look for a natural ambient connection 
$\Nabla$ on $\CE$ which is good enough to give rise the easier interpretation.
We start with a reinterpretation of the construction from \ref{2.4} in terms 
of Weyl structures and conenctions.

\subsection{Partial contact connections}		\label{2.5}
In order to formulate the next results we need the notion of partial contact 
connections.
For a general distribution $D\subset TM$ on a smooth manifold $M$, a {\it 
partial linear connection} on $M$ is an operator $\Ga(D)\times\X(M)\to\X(M)$ 
satisfying the usual conditions for linear connections. 
In other words, we modify the notion of linear connection on $TM$ just
by the requirement to differentiate only in the directions lying in $D$.
If a partial linear connection preserves $D$, then restricting also the second
argument to $D$ yields an operator of the type $\Ga(D)\x\Ga(D)\to\Ga(D)$;
in the case the distribution $D\subset TM$ is contact, we speak about the 
\textit{partial contact connection}.

Given a contact distribution $D\subset TM$ and a classical linear connection
$\nabla$ on $M$, any choice of a contact one--form induces a partial contact
connection $\nabla^D$ as follows.
Let $\th\in\Om^1(M)$ be a contact one--form with the contact subbundle $D$ and 
let $r_\th$ be the corresponding Reeb vector field as in \ref{2.2}.
Let us denote by $\pi_\th:TM\to D$ the bundle projection induced by $\th$,
namely the projection to $D$ in the direction of $\span{r_\th}\subset TM$.
Now for any $X,Y\in\Ga(D)$, the formula 
\begin{equation}
  \nabla^D_XY:=\pi_\th(\nabla_XY)
\end{equation}
defines a partial contact connection and we say that $\nabla^D$ is induced 
from $\nabla$ by $\th$.
The \textit{contact torsion} of the partial contact connection $\nabla^D$
is a tensor field of type $D\wedge D\to D$ 
defined as the projection to $D$ of the classical torsion.
More precisely, if $\nabla^D$ is induced from $\nabla$ by $\th$, $T^D$ 
denotes the contact torsion of $\nabla^D$ and $T$ is the torsion of 
$\nabla$, then
$T^D(X,Y)=\pi_\th(T(X,Y))=\nabla^D_XY-\nabla^D_YX-\pi_\th([X,Y])$ 
for any $X,Y\in\Ga(D)$.

\subsection{General construction revisited}
In the construction of special symplectic connection in \ref{2.4}, 
we started with a choice of an element $a\in\g$ which in particular
induced a contact one--form $\th_a$ on $\Ce_a$.
Then we described the $G_0'$--principal bundle $\Ga_a\to\Ce_a$
which is actually a reduction of the $P$--principal bundle 
$\pi^{-1}(\Ce_a)\to\Ce_a$ to the structure group $G_0'\subset P$.
In terms of subsection \ref{3.2}, the couple $(\pi^{-1}(\Ce_a)\to\Ce_a,\mu)$ 
forms a flat parabolic geometry of type $(G,P)$ and the contact form $\th_a$ 
represents a choice of scale.
In this vein, the reduction $\Ga_a\subset\pi^{-1}(\Ce_a)$ above is
interpreted as (the image of) an exact Weyl structure 
and below we show this is exactly the one 
corresponding to $\th_a$.
In particular, the restriction of the $\g_0'$--part 
of the Maurer--Cartan form $\mu$ to $\Ga_a$ defines the exact Weyl connection
preserving $\th_a$.

Further restriction to a regular subset $U\subset\Ce_a$ and the factorization 
by $T_a$ finally yielded a special symplectic connection on $M_U=T_a\setminus 
U$.
In the current setting together with the definitions in \ref{2.5}, it is
obvious that the resulting connection on $M_U$ is fully determined by the 
partial contact connection induced by $\th_a$ from the exact Weyl connection
on $U\subset\Ce_a$ corresponding to $\th_a$.
Since any Weyl connection preserves the contact distribution $D$, the induced 
partial contact connection is just the restriction to the directions in $D$.
Altogether, we can recapitulate the results in \ref{2.4} as follows. 

\begin{prop*}				\label{th2}
Let $a\in\g$ be so that $\Ce_a\subset\Ce$ is non--empty and let 
$U\subset\Ce_a$ be regular.
Let $\th_a$ be the contact one--form on $U\subset\Ce_a$ determined by $a\in\g$
as in \eqref{eqth}.
Then the special symplectic connection on $M_U$ constructed in \ref{2.4} is 
fully determined by the partial contact connection induced from the exact 
Weyl connection corresponding to $\th_a$.
\end{prop*}
\begin{proof}
According to the discussion above, we only need to show that the exact Weyl
structure represented by $\Ga_a\subset G$ is the one corresponding to the
scale $\th_a$.
This easily follows from the definitions of $\th_a$ and $\Ga_a$:

The contact one--form  $\th_a$ is defined in \eqref{eqth} by $\xi_a$, the
fundamental vector field corresponding to the element $a\in\g$.
The vector field $\xi_a$ on $\Ce_a\subset G/P$ is the projection of the right 
invariant vector field on $\pi^{-1}(\Ce_a)\subset G$ generated by $a$.
Using the identification $T\Ce\cong G\x_P(\g/\p)$, the frame form 
corresponding to $\xi_a$ is the equivariant map $G\to\g/\p$
given by $g\mapsto \Ad_{g^{-1}}(a)+\p$.

On the other hand, the subset $\Ga_a\subset\pi^{-1}(\Ce_a)$ is explicitly described in 
the proof of \cite[Theorem 3.4]{CS} as
$$ \Ga_a=\left\{g\in G:\Ad_{g^{-1}}(a)=\tfrac12 e^2_-+\p'\right\}, $$
where $e^2_-$ is the unique element of $\g_{-2}$ such that $B(e^2_-,e^2_+)=1$.
Obviously, the restriction of the frame form of $\xi_a$ to $\Ga_a$ is
constant, which just means that  the vector field $\xi_a$
is parallel with respect to the exact Weyl connection corresponding to
$\Ga_a$.
Since $\xi_a$ is the Reeb vector field of the contact one--form $\th_a$, 
the latter is parallel if and only if the former is, which completes the proof.
\end{proof}

\subsection{Pull--back connections}		\label{4.0}
Let $p:\CE\to\Ce$ be the cone as in \ref{2.3}.
Any smooth section $s:\Ce\to\CE$ determines a principal connection on 
$\CE$; the corresponding horizontal lift of vector fields is denoted as 
$X\mapsto X^{hor}$.
An ambient linear connection $\Nabla$ on $\CE$ defines a
linear connection $\nabla^s$ on $\Ce$ by the formula
\begin{equation}		\label{eq2}
  \nabla^s_XY:=Tp(\Nabla_{X^{hor}}Y^{hor}).
\end{equation}
We call $\nabla^s$ the \textit{pull--back connection}
corresponding to $s$.
On the other hand, for any section $s$, which we call a choice of scale by
\ref{3.2}, let $\th_s=s^*\Al$ be the contact form and let 
$\bar\nabla^s$ be the corresponding exact Weyl connection on $\Ce$.
In the rest of this section, we are looking for an ambient connection $\Nabla$
on $\CE$ so that both $\nabla^s$ and $\bar\nabla^s$ induce the same partial 
contact connection on $D\subset T\Ce$.
For this reason it turns out that $\Nabla$ has to be symplectic, i.e.
$\Nabla\OM=0$.

The following statement provides together with Theorem \ref{th1} and
Proposition \ref{th2} the desired simple realization of special symplectic 
connections of type (i), (ii), and (iii) according to the list in \ref{2.1}.
The point is  that in all these cases the ambient connection $\Nabla$ is very
natural and easy to describe.

\begin{thm*}			\label{th3}
Let $\CE\to\Ce$ be the model cone for $\g=\frak{sp}(2n+2,\R)$,
$\frak{su}(p+1,q+1)$ or $\frak{sl}(n+2,\R)$.
Then there is an ambient symplectic connection $\Nabla$ on the total space of
$\CE$ so that, for any section $s:\Ce\to\CE$, the induced partial contact 
connections of the exact Weyl connection  and the pull--back connection
corresponding to $s$ coincide.
\end{thm*}

Although the definition of the cone $\CE\to\Ce$ is pretty general,
its convenient interpretation necessary to find a natural candidate for 
$\Nabla$ is no more universal.
In order to prove the Theorem, we deal in following three subsections with
each case individually.
It follows that the reasonable interpretation of the cone in any discussed 
case is more or less standard and we refer primarily to \cite{parabook} for 
a lot of details.
The candidate for an ambient connection $\Nabla$ is almost canonical, 
therefore in the proofs of subsequent Propositions we focus only in the 
justification of the choices.

Note that a natural guess for $\Nabla$ to be a $G$--invariant
symplectic connection on $\CE=G/P'$ does help only for contact projective
structures, i.e. the structures corresponding to the contact grading of 
$\g=\frak{sp}(2n+2,\R)$. 
This is due to the following statement, which is an immediate corollary of 
\cite[Theorem 3]{P-T}:
For a connected real simple Lie group $G$ with Lie algebra $\g$,
the nilpotent adjoint orbit $\Ce=\Ad_G(e^2_+)$ admits a $G$--invariant linear 
connection if and only if $\g\cong\frak{sp}(m,\R)$.

For a reader's convenience we assume the dimension of $\Ce=G/P$ to be always 
$m=2n+1$.
Consequently, $\dim\CE=2n+2$ and we further continue the convention that all 
important objects on $\CE$ are denoted with the hat.

\subsection{Contact projective structures}	\label{4.1}
Contact projective structures correspond to the contact grading of 
the Lie algebra $\g=\frak{sp}(2n+2,\R)$, the only real form of
$\frak{sp}(2n+2,\C)$ admitting the contact grading.
These structures are studied in \cite{Fox} in whole generality:
\textit{contact projective structure} on a contact manifold $(M,D)$ is defined 
as a contact path geometry such that the paths are among  geodesics of a linear
connection on $M$; the paths are then called \textit{contact geodesics}.
In analogy to classical projective structures, a contact projective structure
is given by a class of linear connections $[\nabla]$ on $TM$ having the same
contact torsion and the same non--parametrized geodesics such that the 
following property is satisfied: 
if a geodesic is tangent to $D$ in one point then it remains tangent to 
$D$ everywhere.

The model contact projective structure is observed on the projectivization
of symplectic vector space $(\R^{2n+2},\OM)$  with $\OM$ being a standard
symplectic form.
Let $G$ be  the group of linear automorphisms of $\R^{2n+2}$ 
preserving $\OM$, i.e. $G:=Sp(2n+2,\R)$.
In order to represent conveniently the contact grading of the corresponding Lie algebra, 
let $\OM$ be given by the matrix
$ \pmat{0&0&1\\0&\Bbb J&0\\-1&0&0}$,
with respect to the standard basis of $\R^{2n+2}$, where 
$\Bbb J=\pmat{0&\Bbb I_n\\-\Bbb I_n&0}$ and $\Bbb I_n$ is the identity matrix 
of rank $n$.
For $\Bbb J^t=-\Bbb J$, the Lie algebra $\g=\frak{sp}(2n+2,\R)$ is represented 
by block matrices of the form
$$
\g=\left\{\pmat{a&Z&z\\X&A&\Bbb JZ^t\\x&-X^t\Bbb
	J&-a}:A\in\frak{sp}(2n,\R)\right\},
$$
where the non--specified entries are arbitrary, i.e. 
$x,a,z\in\R$, $X\in\R^{2n}$ and $Z\in\R^{2n*}$, and the fact 
$A\in\frak{sp}(2n,\R)$ means that $A^t\Bbb J+\Bbb JA=0$.
Particular subspaces of the contact grading 
$\g_{-2}\oplus\g_{-1}\oplus\g_0\oplus\g_1\oplus\g_2$ of $\g$ 
is read along the diagonals so that $\g_{-2}$ is represented by $x\in\R$, 
$\g_{-1}$ by $X\in\R^{2n}$, etc.
In particular, $\g_0$ is represented by the pairs 
$(a,A)\in\R\x\frak{sp}(2n,\R)$ so that 
$\frak{sp}(2n,\R)$ is the semisimple part $\g_0^{ss}$ and the center 
$\frak{z}(\g_0)$ is generated by the grading element $E$ corresponding to the
pair $(1,0)$.
Following the general setup in \ref{2.2}, 
$\p=\g_0\oplus\g_1\oplus\g_2$, $\p'=\g_0^{ss}\oplus\g_1\oplus\g_2$, 
and $P,P'$ are the corresponding connected Lie subgroups in $G$.
Schematically, the parabolic subgroup $P\subset G$ is given as
$$
P=\left\{\pmat{r&*&*\\0&*&*\\0&0&r^{-1}}:r\in\R_+\right\}
$$
and $P'\subset P$ corresponds to $r=1$.
Easily, $G$ acts transitively on $\R^{2n+2}\bez0$, $P'$ is the stabilizer of the
first vector of the standard basis, and $P$ is the stabilizer of the
corresponding ray.
Hence $\CE\cong G/P'$ is identified with $\R^{2n+2}\bez0$ and its oriented 
projectivization $\Ce\cong G/P$ is further identified with the sphere
$S^{2n+1}\subset\R^{2n+2}$.
Altogether, we have interpreted the model cone for contact projective 
structures as 
$$ \CE\cong\R^{2n+2}\bez0\to S^{2n+1}\cong\Ce. $$
It is easy to check that 
the canonical symplectic form on $\CE$ corresponds to the standard symplectic
form on $\R^{2n+2}$, which is $G$--invariant by definition.
As a particular interpretation of the general definition in \ref{2.3},
the contact distribution $D\subset TS^{2n+1}$ 
is given by $D_v=v^\perp\cap T_vS^{2n+1}$, where $v\in S^{2n+1}$ and
$v^\perp=\{x\in\R^{2n+2}:\OM(v,x)=0\}$.

Next, let $\Nabla$ be the canonical flat connection on $\R^{2n+2}$.
Then the connections on $S^{2n+1}$ defined by \eqref{eq2} form  projectively
equivalent connections having the great circles as common non--parametrized
geodesics.
Any great circle is the intersection of $S^{2n+1}$ with a plane passing 
through 0. 
If the plane is isotropic with respect to $\OM$, we end up with contact 
geodesics.
Note that no  connection in the class preserves the contact
distribution, since it is obviously torsion--free, however the induced partial
contact connection coincides with the restriction of an exact Weyl connection to
$D$:

\begin{prop*}
Let $\CE\to\Ce$ be the model cone for $\g=\frak{sp}(2n+2,\R)$. 
Then $\Ce\cong S^{2n+1}$, $\CE\cong\R^{2n+2}\bez0$, $\OM$ corresponds to the 
standard symplectic form on $\R^{2n+2}$, and the ambient symplectic 
connection $\Nabla$ from Theorem \ref{th3} is the canonical flat connection 
on $\R^{2n+2}$.
\end{prop*}
\begin{proof}
Since $\CE\cong G/P'$, the tangent bundle $T\CE$ is identified with the
associated bundle $G\x_{P'}(\g/\p')$ via the Maurer--Cartan form $\mu$ on $G$,
where the action of $P'$ on $\g/\p'$ is induced from the adjoint representation.
On the other hand, $\CE\cong\R^{2n+2}\bez0$, so $\g/\p'\cong\R^{2n+2}$ as 
vector spaces. 
$\R^{2n+2}$ is the standard representation of $G$ and 
the essential observation for the next development is that its restriction to 
$P'\subset G$ is isomorphic to the representation of $P'$ on $\g/\p'$.
Explicitly, the isomorphism $\R^{2n+2}\to\g/\p'$ is given by 
\begin{equation}			\label{eq4.1}
\pmat{a\\X\\x}\mapsto \pmat{a&0&0\\X&0&0\\x&-X^t\Bbb J&-a}+\p'.
\end{equation}
Altogether, $T\CE\cong G\x_{P'}\R^{2n+2}$ and since the representation of $P'$
is the restriction of a representation of $G$, the Maurer--Cartan form $\mu$
induces a linear connection on $T\CE$ by general arguments as in \cite{CG}.
More precisely, $G\x_{P'}\R^{2n+2}\cong(G\x_{P'}G)\x_G\R^{2n+2}$, where
the (homogeneous) principal bundle $G\x_{P'}G\to\CE$ represents the symplectic 
frame bundle of $\CE$.
The Maurer--Cartan form $\mu$ on $G$ extends to a $G$--invariant principal connection on 
$G\x_{P'}G$.
The latter connection induces connections on all associated bundles, in
particular, this gives rise to a flat invariant symplectic connection on $T\CE$,
i.e. the canonical flat connection $\Nabla$ on $\R^{2n+2}$.
Due to this interpretation of $\Nabla$, we are going to describe the covariant
derivative with respect to $\Nabla$ in an alternative way which will provide a
comparison of the pull--back and exact Weyl connections.

For a vector field $\hat X\in\X(\CE)$ let us denote by $f_{\hat X}$ the corresponding 
frame form,  i.e. the $P'$--equivariant map from  $G$ to $\g/\p'\cong\R^{2n+2}$.
As $\Nabla$ is an instance of tractor connection, 
the frame form of the covariant derivative of $\hat Y$ in the direction of 
$\hat X$ turns out to be expressed as
\begin{equation}		\label{eq4.2}
  f_{\Nabla_{\hat X}\hat Y} =\hat\xi\. f_{\hat Y} +\mu(\hat\xi)\o f_{\hat Y},
\end{equation}
where $\hat\xi\in\X(G)$ is a lift of $\hat X\in\X(\CE)$ and
$\o$ denotes the standard representation of $\g$ on $\R^{2n+2}$; 
see \cite[section 2]{CG} or \cite[section 2.15]{Slo}.

From now on, let $s:\Ce\to\CE$ be  a fixed section of the model cone, i.e. 
a choice of scale, and let $\si^s:\G_0'\to G$ be the corresponding exact Weyl
structure, where $\G_0'$ is the principal $G_0'$--bundle as in \ref{3.2}.
Since $\R^{2n+2}\cong\g_-\oplus\span E$ as $G'_0$--modules, the section $s$ 
provides the identification $T\CE\cong\G_0'\x_{G_0'}(\g_-\oplus\span E)$
(similarly, $T\Ce\cong\G_0'\x_{G_0'}\g_-$).
If $\hat X$ is a vector field on $\CE$ and $f_{\hat X}$ the corresponding 
frame form as above, than the frame form corresponding to the
identification $T\CE\cong\G_0'\x_{G_0'}(\g_-\oplus\span E)$ is given by
$f_{\hat X}\o\si^s$. (Similarly for vector fields on $\Ce$.)
Restricting to the image of $\si^s$ within $G$, we do not distinguish between
these two interpretations.

In the definition of the pull--back connection, 
$X^{hor}\in\X(\CE)$ denotes the  horizontal lift of vector field 
$X\in\X(\Ce)$ with respect to the principal connection on $\CE$ determined by $s$.
According to the identifications above, the horizontality in terms of the frame forms 
is expressed as $f_{X^{hor}}=(0,f_X)^t\in\R^{2n+2}\cong\span E\oplus\g_-$.
Hence the formula \eqref{eq4.2} yield
\begin{equation}			\label{eq4.3}
  f_{\Nabla_{X^{hor}}Y^{hor}} =\pmat{0\\\hat\xi\. f_Y} +\mu(\hat\xi)\o\pmat{0\\f_Y}.
\end{equation}
The tangent map of the projection $p:\CE\to\Ce$ corresponds to the projection
$\pi:\g_-\oplus\span E\to\g_-$ in the direction of $\span E$,
hence the result of the covariant derivative $\nabla^s_XY$ with respect to the 
pull--back connection defined by \eqref{eq2} corresponds to the $\g_-$ part 
of \eqref{eq4.3}.

On the other hand, the covariant derivative $\bar\nabla^s$ with respect to the exact 
Weyl connection corresponding to $s$ is given by 
$f_{\bar\nabla^s_XY} =\xi^s\.(f_Y\o\si^s)$, where $\xi^s\in\X(\G_0')$ is the
horizontal lift of $X\in\X(\Ce)$ with respect to the principal connection on
$\G_0'$.
This is characterized by $\mu_0(T\si^s\.\xi^s)=0$, i.e. $T\si^s\.\xi^s=\xi+\ze_{\Rho^s(\xi)}$  
where $\xi$ is the lift such that $\mu(\xi)\in\g_-$ and $\Rho^s$ is the Rho--tensor.
Since $\xi^s\.(f_Y\o\si^s)=(T\si^s\.\xi^s)\.f_Y$, we conclude by the formula
\begin{equation}		\label{eq4.4}
  f_{\bar\nabla^s_XY} =\xi\. f_Y -\ad(\Rho^s(\xi))(f_Y).
\end{equation}

Altogether, considering $T\si^s\.\xi^s$ instead of $\hat\xi$ in \eqref{eq4.3},
the desired comparison of the pull--back connection and the exact 
Weyl connection determined by $s$ is given by 
\begin{equation}		\label{eq4.5}	
  f_{\nabla^s_XY-\bar\nabla^s_XY} =\pi\left((\mu(\xi)+\Rho^s(\xi))\o\pmat{0\\f_Y}\right),
\end{equation}
where $\pi$ denotes the projection $\g_-\oplus\span E\to\g_-$ as before.
In particular, expressing the standard action on the right
hand side of \eqref{eq4.5} for $X,Y\in\Ga(D)$, i.e. for $f_X,f_Y\in\g_{-1}$, 
the difference tensor turns out to be of the form
\begin{equation}		
   \nabla^s_XY-\bar\nabla^s_XY =-d\th_s(X,Y) r_s
\end{equation}
where $\th_s$ and $r_s$ is the contact form and the Reeb vector field,
respectively, corresponding to the scale $s:\Ce\to\CE$.
This shows that the induced partial contact connections of the 
pull--back connection and the exact Weyl connection determined by $s$ coincide.
\end{proof}

\subsection*{Remarks}
(a) 
The paper \cite{CGS} provides a characterization of symplectic connections 
of Ricci type with specific symplectic connections obtained by a reduction 
procedure from a hypersurface in  a symplectic vector space.
More specifically, for $a\in\g=\frak{sp}(2n+2,\R)$ the hypersurface
in $\R^{2n+2}$ is defined by 
$$ 
\Sigma_a:=\{x\in\R^{2n+2}:\OM(x,ax)=1\},
$$ 
where $\OM$ is the standard symplectic form, and
all the connections are induced from the flat ambient connection on 
$\R^{2n+2}$.
Basically, this is just another view on the description of
pull--back connections which is conceivable whenever $\Ce$ can be interpreted 
as a hypersurface in $\CE$; the section $s:\Ce\to\CE$ is then understood as 
a deformation of the hypersurface.
In the current case, $\Ce\cong S^{2n+1}\subset\R^{2n+2}\bez0\cong\CE$ and
one easily shows that for the section $s_a$ corresponding to an 
element $a\in\g$, the image of $s_a$ really coincides with the hypersurface 
$\Sigma_a$ above.

(b) Note that the argument in the proof of Proposition above can be 
directly generalized in at least two ways:
First, the homogeneous model and the flat connection $\Nabla$ can be replaced
by a general manifold $M$ with contact projective structure and the unique ambient 
connection on the total space of a scale bundle over $M$, respectively, 
which is established in \cite[Theorem B]{Fox}.
The general ambient connection is induced by a canonical Cartan connection in 
the very same manner as above.
Second, the comparison of pull--back connections and exact Weyl connections can 
be extended to general Weyl connections. 
Indeed, any Weyl connection corresponds by \ref{3.2} to a principal connection
on a scale bundle, which actually is the important ingredient in the 
definition of pull--back connections in \eqref{eq2}.
The fact that the principal connection on the bundle of scales is given by a 
section plays no role in this context.

(c) By \cite[Theorem A]{Fox}, any choice of scale determines a unique linear 
connection on $M$ so that it preserves the corresponding contact 
form and its differential, represents the contact projective structure, and 
has a normalized torsion.
Note that this is neither the pull--back connection nor the exact Weyl
connection, however the induced partial contact connection is still the same.
Connections of this type are close analogies of Webster--Tanaka connections
well known in  CR geometry.

\subsection{CR structures of hypersurface type}		\label{4.2}
These structures correspond to the contact grading of the
Lie algebra $\g=\frak{su}(p+1,q+1)$, a real form of $\frak{sl}(n+2,\C)$, 
where $p+q=n$ once for all.
In fact the correct full name of the general geometric structure of this type
is \textit{non--degenerate partially integrable almost CR structure of
hypersurface type}.
This structure on a smooth manifold $M$ is given by a contact distribution 
$D\subset TM$ with a complex structure $J:D\to D$ so that the Levi bracket
$\L:D\wedge D\to TM/D$ is compatible with the complex structure, i.e.
$\L(J-,J-)=\L(-,-)$ for any $-,-\in\Ga(D)$.
A~choice of contact form provides an identification of $T_xM/D_x$ with $\R$, 
for any $x\in M$, and the latter condition on the
Levi bracket says that $\L(-,J-)$ is a non--degenerate symmetric bilinear 
form on $D$, that is a pseudo--metric.
Hence $\L(-,J-)+i\L(-,-)$ is  a Hermitean form on $D$ whose signature $(p,q)$
is the \textit{signature} of the CR structure.

The classical examples of CR structures of the above type are induced on 
non--degenerate real hypersurfaces in $\C^{n+1}$.
In general, for a real submanifold $M\subset\C^{n+1}$, the CR structure on 
$M$ is induced from the ambient complex space $\C^{n+1}$ so that
the distribution $D$ is the maximal complex subbundle in $TM$, and the complex structure $J$
is the restriction to $D$ of the multiplication by $i$.
The model CR structures of hypersurface type are induced on the so--called
\textit{hyperquadrics}, cf. \cite{Jac}.
A~typical hyperquadric of signature $(p,q)$ is described as a graph
\begin{equation}		\label{eq3}
\Cal Q:=\{(z,w)\in\C^n\x\C:\Im(w)=h(z,z)\},
\end{equation}
or as
\begin{equation}		\label{eq4}
\Cal S:=\{(z,w)\in\C^n\x\C:h(z,z)+|w|^2=1\},
\end{equation}
where $h$ is  a Hermitean form of signature $(p,q)$.
It turns out that the induced CR structures on $\Cal Q$ and $\Cal S$ are 
equivalent and the equivalence is established by the restriction of the 
biholomorphism 
$(z,w)\mapsto\left(\frac{z}{w-i}, \frac{1-iw}{w-i}\right)$.
Note that this identification is almost global (only the point $(0,i)\in\Cal S$
is mapped to infinity) and projective.
In particular, $\Cal Q$ and $\Cal S$ are different affine realizations of a 
projective hyperquadric in $\C\P^{n+1}$ which  is identified with the 
homogeneous space $G/P$ as follows.

Let $G$ be the group of complex linear automorphisms of $\C^{n+2}$ preserving
a Hermitean form $H$ of signature $(p+1,q+1)$, i.e. $G:=SU(p+1,q+1)$.
Let the Hermitean form $H$ be given by the 
matrix $\pmat{0&0&-\frac i2\\0&\Bbb I&0\\\frac i2&0&0}$, with respect to the standard 
basis  $(e_0,e_1,\dots,e_n,e_{n+1})$, where 
$\Bbb I=\pmat{\Bbb I_p&0\\0&-\Bbb I_q}$ represents the Hermitean form $h$ of 
signature $(p,q)$ on $\span{e_1,\dots,e_n}\subset\C^{n+2}$.
According to this choice,
the Lie algebra $\g=\frak{su}(p+1,q+1)$ is represented by matrices
of the following form with blocks of sizes 1, $n$, and 1
$$
\g=\left\{\pmat{c&2iZ&v\\X&A&\Bbb I\bar Z^t\\u&-2i\bar X^t\Bbb I&-\bar c}:
u,v\in\R, %%X\in\C^n, Z\in\C^{n*}, c\in\C, 
A\in\frak{u}(p,q), \tr(A)+2i\Im(c)=0 \right\},
$$
where the non--specified entries are arbitrary, i.e. $X\in\C^n, Z\in\C^{n*}$, 
and $c\in\C$.
(Note that $A\in\frak{u}(p,q)$ means $\bar A^t\Bbb I+\Bbb IA=0$, so in
particular $\tr(A)$ is purely imaginary complex number.)
The contact grading of $\g$ is read along the diagonals as in \ref{4.1}.
In particular, $\g_0$ is represented by the pairs $(c,A)\in\C\x\frak{u}(p,q)$
with the constrain $\tr(A)+2i\Im(c)=0$.
The center $\frak{z}(\g_0)$ is two--dimensional, where the grading element 
$E$ corresponds to the pair $(1,0)$, and the semisimple part $\g_0^{ss}$ is
isomorphic to $\frak{su}(p,q)$.
The subalgebra $\g_0'\cong\frak{u}(p,q)$ corresponds to the pairs of the form 
$(-\frac12\tr(A),A)$.
Note that the compatibility of the Levi bracket with the complex structure on $D$
is reflected here by the fact that $[iX,iY]=[X,Y]$ for any  $X,Y\in\g_{-1}$.
Subalgebras $\p'\subset\p\subset\g$ are defined as in \ref{2.2}, $P'\subset P$ 
are the corresponding connected subgroups in $G$.
The parabolic subgroup $P\subset G$ is schematically indicated as
$$
P=\left\{\pmat{re^{i\th}&*&*\\0&*&*\\0&0&\frac1re^{i\th}}: r\in\R_+\right\}
$$
and $P'\subset P$ corresponds to $r=1$.

Let $\Cal N$ be the set of non--zero null--vectors in $\C^{n+2}$ with respect to the
Hermitean form $H$.
Clearly, $G$ preserves and acts transitively on $\Cal N$.
If $Q\subset G$ denotes the stabilizer of the first vector of the standard 
basis then $\Cal N$ is identified with the homogeneous space $G/Q$.
Obviously $Q\subset P'\subset P$ corresponds to $r=1$ and $\th=0$ according to
the description of $P$ above.
Since $P'/Q\cong U(1)$, the group of complex numbers of unit
length, the homogeneous space $G/P'$ is identified with $\Cal N/U(1)$.
Next $P\supset P'$ is the stabilizer of the complex line generated by the 
first vector of the standard basis,
so the homogeneous space $G/P$ is identified with $\Cal N/\C^*$, the
complex projectivization of $\Cal N$.
Altogether a natural  interpretation of the model cone in this case is
$$ \CE\cong\Cal N/U(1)\to\Cal N/\C^*\cong\Ce. $$

A~direct substitution shows that the hyperquadric $\Cal Q$ from \eqref{eq3}
is the intersection of $\Cal N$ with the complex hyperplane $z_0=1$.
According to the new basis $(e_0+ie_{n+1},e_1,\dots,e_n,e_0-ie_{n+1})$
of $\C^{n+2}$, the Hermitean metric $H$ is in the diagonal form
%$|z'_{n+1}|^2-|z'_0|^2+h(z',z')$ 
so that the hyperquadric $\Cal S$ from \eqref{eq4} is the intersection of 
$\Cal N$ with the complex hyperplane $z'_0=1$ (where the dash refers to
coordinates with respect to the new basis).
This recovers the identification above, in particular, both $\Cal Q$ and
$\Cal S$  are identified with $\Cal N/\C^*\cong\Ce$.

From now on, let $\Ce$ be the hyperquadric $\Cal S$ in the hyperplane $z'_0=1$
which we naturally identify with $\C^{n+1}$.
This hyperplane without the origin is further identified with $\Cal
N/U(1)\cong\CE$ under the map $(z',w')\mapsto(\sqrt{|h(z',z')+|w'|^2|},z',w')$.
Denote by $\h$ the induced Hermitean metric (of signature $(p+1,q)$) on this
hyperplane and let $\OM$ be its imaginary part. 
Obviously, both $\h$ and $\OM$  are $G$--invariant, and an easy calculation
shows that $\OM$ corresponds to the canonical symplectic form on $\CE$ up to
non--zero constant multiple.
Altogether, the defining equation \eqref{eq4} for  $\Cal S\subset\C^{n+1}$ 
reads as 
\begin{equation}			\label{eq5}
\Cal S=\{z\in\C^{n+1}:\h(z,z)=1\}
\end{equation}
and the most satisfactory interpretation of the model cone is
$$
\CE\cong\C^{n+1}\bez0\to\Cal S\cong\Ce.
$$

\begin{prop*}
Let $\CE\to\Ce$ be the model cone for $\g=\frak{su}(p+1,q+1)$. 
Then $\CE\cong\C^{n+1}\bez0$ and $\Ce\cong\Cal S$, the hyperquadric in 
$\C^{n+1}\bez0$ given by \eqref{eq5}, where $\h$ is the Hermitean metric of 
signature $(p+1,q)$.
Further, $\OM$ corresponds to the imaginary part of $\h$ and the ambient 
symplectic connection $\Nabla$ from Theorem \ref{th3} is the canonical flat 
connection on $\C^{n+1}$.
\end{prop*}

\begin{proof}
The connection $\Nabla$ is obviously symplectic, i.e. $\Nabla$ is 
torsion--free and $\Nabla\OM=0$.
By definition, $\OM$ is the imaginary part of the Hermitean metric $\h$ on
$\C^{n+1}$.
Its real part $\hat g$ is then expressed in terms of $\OM$ and the standard 
complex structure on $\C^{n+1}$ as $\hat g=\OM(-,i-)$.
This is a real pseudo--metric on $\C^{n+1}\cong\R^{2n+2}$ of signature $(2p+2,2q)$
and $\Nabla$ can be seen as the Levi--Civita connection of $\hat g$.

As in general, let $\Al:=\E\ins\OM$.
Let $s:\Cal S\to\C^{n+1}\bez0$ be  a section of the cone and let 
$\th:=s^*\Al$ be the corresponding contact one--form on $\Cal S$.
Then $g:=d\th(-,i-)$ is a non--degenerate symmetric bilinear form on the
contact distribution $D$ which has to be preserved by the Weyl connection 
$\bar\nabla^s$. 
Next, since we deal with the homogeneous model, the contact torsion of 
$\bar\nabla^s$ vanishes.
In fact, the corresponding partial contact connection on $D$ is 
uniquely determined by the fact that 
(i) it leaves $g$ to be parallel and 
(ii) its contact torsion vanishes.

In order to prove the statement, it suffices to show that (i) and (ii) is
satisfied also by the partial contact connection induced by 
the pull--back connection $\nabla^s$ corresponding to $s$.
However, since $\Nabla$ is torsion--free, the pull--back connection $\nabla^s$ 
is torsion--free as well, hence the condition (ii) is satisfied trivially.
The condition (i) follows as follows: 
For $X,Y,Z\in \Ga(D)$, expand
$$
(\nabla^s_X g)(Y,Z) = X\.d\th(Y,iZ)-d\th(\nabla^s_XY,iZ)-d\th(Y,i\nabla^s_XZ).
$$
Since $\th=s^*\Al$ and $d\Al=\OM$, by the naturality of exterior differential we
have got $d\th=s^*\OM$.
Next easily, $Ts\.X=X^{hor}\o s$ and, by the definition of the pull--back 
connection in \ref{4.0}, 
$Ts\.\nabla^s_XY=\Nabla_{X^{hor}}Y^{hor}\o s\mod\span\E$.
Since $\Al=\frac12\E\ins\OM$ and $D=Tp\.\ker\Al$, the previous formula is 
rewritten as
\begin{eqnarray*}
X\.\OM(Ts\.Y,Ts\.iZ)-\OM(Ts\.\nabla^s_XY,Ts\.iZ)-\OM(Ts\.Y,Ts\.i\nabla^s_XZ)=\\
	=X^{hor}\.\OM(Y^{hor},iZ^{hor})-\OM(\Nabla_{X^{hor}}Y^{hor},
		iZ^{hor})-\OM(Y^{hor},i\Nabla_{X^{hor}}Z^{hor}).
\end{eqnarray*}
However, the very last expression is just $(\Nabla_{X^{hor}}\hat g)(Y^{hor},
Z^{hor})$, which vanishes trivially by definitions.
\end{proof}

\subsection{Lagrangean contact structures}	\label{4.3}
Lagrangean contact structures correspond to the contact grading of 
$\g=\frak{sl}(n+2,\R)$, another real form of $\frak{sl}(n+2,\R)$.
\textit{Lagrangean contact strucure} on a smooth manifold $M$ consists of the 
contact distribution $D\subset TM$ and a fixed decomposition $D=L\oplus R$ so 
that the subbundles $L$ and $R$ are \textit{Lagrangean}, i.e. isotropic
with respect to the Levi bracket $\L:D\wedge D\to TM/D$.
These structures was profoundly studied in \cite{Tak} where we refer for a lot
of details.
The model Lagrangean contact structure appears on the projectivization of
the cotangent bundle of real projective space; let us present the algebraic 
background first.

The contact grading of $\g=\frak{sl}(n+2,\R)$ is read diagonally as in 
\ref{4.1} and \ref{4.2} from the following block decomposition
$$
\g=\left\{\pmat{a&Z_1&z\\X_1&B&Z_2\\x&X_2&c}:a+\tr(B)+c=0\right\},
$$
where as usual the non--specified entries are arbitrary, i.e. $x,a,c,z\in\R$,
$X_1,Z_2\in\R^n$, $X_2,Z_1\in\R^{n*}$, and $B\in\frak{gl}(n,\R)$.
The subalgebra $\g_0$ is represented by the triples
$(a,B,c)\in\R\x\frak{gl}(n,\R)\x\R$ so that $a+\tr(B)+c=0$.
The center $\frak{z}(\g_0)$ is two--dimensional and the grading element $E$
corresponds to $(1,0,-1)$.
The semisimple part $\g_0^{ss}$ is isomorphic to $\frak{sl}(n,\R)$ and the
subalgebra $\g_0'\cong\frak{gl}(n,\R)$  is represented by all triples of the 
form $\left(-\frac12\tr(B),B,-\frac12\tr(B)\right)$.
The subspace $\g_{-1}$ defining the contact distribution is split as
$\g_{-1}=\g_{-1}^L\oplus\g_{-1}^R$, where $\g_{-1}^L$ is represented by
$X_1\in\R^n$ and $\g_{-1}^R$ by $X_2\in\R^{n*}$, so that 
this splitting  is invariant under the adjoint action of $\g_0$.
Furthermore, the subspaces $\g_{-1}^L$ and $\g_{-1}^R$ are isotropic  with
respect to the bracket $[\ ,\  ]:\g_{-1}\x\g_{-1}\to\g_{-2}$, which reflects 
the geometric definition of the structure in terms of Levi bracket.
Similarly, $\g_1$ splits as $\g_1^L\oplus\g_1^R$.
The subalgebras $\p'\subset\p\subset\g$ are given as before.
Let $G$ be the group $SL(n+2,\R)$.
The connected parabolic subgroup $P\subset G$ corresponding to $\p\subset\g$ is
schematically indicated as 
$$
P=\left\{\pmat{pq&*&*\\0&*&*\\0&0&\frac pq}:p,q\in\R_+\right\}
$$
and $P'\subset P$ corresponds to $q=1$.

The homogeneous space $G/P$ is naturally identified with the set of 
flags of half--lines in hyperplanes in $\R^{n+2}$.
Indeed, the standard action of $G$ on $\R^{n+2}$ descends to a transitive
action both on rays and hyperplanes in $\R^{n+2}$, so $G$ acts transitively on
the set of flags of above type.
The subgroup $P$ is the stabilizer of the flag $\ell\subset\rho$ where $\ell$
and $\rho$ is the  ray and the hyperplane generated by the first and the first 
$n+1$ vectors from the standard basis, respectively.
Obviously, $P=\tilde P\cap \bar P$ where $\tilde P$ is the stabilizer of 
$\ell$ and $\bar P$ stabilizes $\rho$.
Note that both $\tilde P$ and $\bar P$ are also parabolic.

We claim that $G/P\cong\Proj^o(T^*S^{n+1})$ which is the oriented
projectivization of the cotangent bundle of projective sphere, the
oriented projectivization of $\R^{n+2}$.
This can be clarified as follows:
The projective sphere $S^{n+1}\cong\Proj^o(\R^{n+2})$ is identified with
$G/\tilde P$, where $\tilde P\subset G$ is the stabilizer of the ray $\ell$ 
as above.
Let $\tilde\p\subset\g$ be the Lie algebra of $\tilde P$ and let
$\g=\tg_{-1}\oplus\tg_0\oplus\tg_1$ be the corresponding grading of $\g$.
As usual, $(\g/\tp)^*\cong\tg_{-1}^*\cong\tg_1$, hence $T^*S^{n+1}\cong
T^*(G/\tilde P)$ is identified with $G\x_{\tilde P}\tg_1$ via the 
Maurer--Cartan form on $G$.
Now, the adjoint action of $\tilde P$ on $\tg_1$ is transitive and an easy
direct calculation shows that the stabilizer of a convenient element of 
$\tg_1$ is precisely $P'\subset P\subset\tilde P$;
the subgroup $P\subset\tilde P$ is the stabilizer of the corresponding ray.
Altogether, $\tg_1\cong\tilde P/P'$ and $\Proj^o(\tg_1)\cong\tilde P/P$,
so $T^*S^{n+1}\cong G/P'$ and $\Proj^o(T^*S^{n+1})\cong G/P$.
Hence the interpretation of the model cone for Lagrangean contact structures 
is
$$
\CE\cong T^*S^{n+1}\to\Proj^o(T^*S^{n+1})\cong\Ce
$$
so that the canonical $G$--invariant symplectic form on $\CE$ corresponds to
the canonical symplectic form on the cotangent bundle $T^*S^{n+1}$, cf. remark
\ref{2.3}(c).

Now we are going to expose a general construction following \cite{Tak}; 
it turns out this will be useful to find a candidate for the ambient connection $\Nabla$ on $\CE\cong T^*S^{n+1}$.
Let $M$ be a manifold with linear torsion--free connection $\nabla$ and 
let $H\subset TT^*M$ be the corresponding horizontal distributions on the 
cotangent bundle over $M$.
Together with the vertical subbundle $V$ of the projection $p:T^*M\to M$ we
have got an almost product structure on $T^*M$.
Let $\Al$ be the canonical one--form and $\OM=d\Al$ the canonical
symplectic form on $T^*M$.
By definition of $\OM$, the subbundle $V$ is isotropic with respect to $\OM$.
The complementary subbundle $H$ determined by the connection $\nabla$ is
isotropic if and only if $\nabla$ is torsion--free.
After the projectivization, the decomposition $V\oplus H=TT^*M$ yields a 
Lagrangean contact structure on $\Proj(T^*M)$. 
Moreover, the almost product structure on $T^*M$ and so the Lagrangean contact
structure on $\Proj(T^*M)$ are independent on the choice of connection from 
the projectively equivalent class $[\nabla]$.
Altogether, starting with a projective structure on a smooth manifold $M$, 
this gives rise to a Lagrangean contact structure on the projectivized
cotangent bundle of $M$.
Note that in terms of parabolic geometries, this construction is an instance 
of the so--called correspondence space construction \cite[section 4]{C2} which is
formally powered by the inclusion $P\subset\tilde P$ of parabolic subgroups in
$G$. 
As a particular implementation of a general principle, locally flat projective 
structure on $M$ gives rise  to a locally flat Lagrangean contact structure 
on $\Proj(T^*M)$.
This is actually observed elementarily in the previous paragraph provided we 
consider oriented projectivization instead of the usual one.

\begin{prop*}
Let $\CE\to\Ce$ be the model cone for $\g=\frak{sl}(n+2,\R)$. 
Then $\CE\cong T^*S^{n+1}$, $\Ce\cong\Proj^o(T^*S^{n+1})$, and $\OM$
corresponds to the canonical symplectic form on cotangent bundle.
Let further $J:TT^*S^{n+1}\to TT^*S^{n+1}$ be the almost product structure 
given by the projective structure on $S^{n+1}$ as above.
Then the bilinear form $\hat g:=\OM(-,J-)$ on $T^*S^{n+1}$ is symmetric and 
non--degenerate and the ambient  symplectic connection $\Nabla$ from Theorem 
\ref{th3} is the Levi--Civita connection of $\hat g$.
\end{prop*}
\begin{proof}
Let $S^{n+1}\subset\R^{n+2}$ be the standard projective sphere.
The projective structure $[\nabla]$ is induced from the canonical flat
connection in $\R^{n+2}$, in particular, any connection in the class is
torsion--free. 
As before, this ensures that both subbundles $V$ and $H$ from the
corresponding decomposition of $TT^*S^{n+1}$ are isotropic with respect to the
canonical symplectic form $\OM$.
The decomposition $V\oplus H=TT^*S^{n+1}$ determines the product structure $J$ 
so that $V$ and $H$ is the eigenspace of $J$ corresponding to the eigenvalue 
$1$ and $-1$, respectively.
Since both $\OM$ and $J$ are non--degenerate, the same holds true also for 
$\hat g:=\OM(-,J-)$. 
Since both $V$ and $H$ are isotropic with respect to $\OM$, the bilinear form $\hat g$ 
turns out to be symmetric, hence it is a pseudo--metric on $T^*S^{n+1}$.

The rest of the proof is completely parallel to that in \ref{4.2} up to the
interchange between the almost complex and almost product structure on $\CE$
and $D\subset T\Ce$, respectively.
\end{proof}

%%%%%%%%%%%%%%%%%%%%%%%%%%%%%%%%%%%%%%%%%%%%%%%%%%%%%%%%%%%%%%%%%%%%%%%%

\end{document}